\documentclass[a4paper,
fontsize=11pt,%
oneside,%
numbers=enddot]{scrartcl}
\KOMAoptions{DIV=12}    
\usepackage[utf8]{inputenc}
\usepackage[T1]{fontenc}
\usepackage{graphicx}
\usepackage{textcomp}
\usepackage[fleqn]{amsmath}
\usepackage{amsthm}
\usepackage{amssymb}
\usepackage{amsfonts}
\usepackage{soul}
\usepackage{pifont}        
\usepackage{libertine}
\usepackage[libertine,
slantedGreek,
nosymbolsc,
nonewtxmathopt,
subscriptcorrection]{newtxmath}
\usepackage[scaled=0.95,varqu,varl]{inconsolata}
\usepackage[scr=boondox]{mathalpha}   
\usepackage{euscript}   
\usepackage{leftindex}
\allowdisplaybreaks[1]  
\numberwithin{equation}{section}    
\usepackage[svgnames,hyperref]{xcolor}
\definecolor{orng}{HTML}{F35400}
\definecolor{bleu}{HTML}{BCE6F2}
\definecolor{dblue}{HTML}{0455BF}
\definecolor{dgreen}{HTML}{02724A}
\definecolor{dgreen2}{HTML}{025951}
\definecolor{dred}{HTML}{D90404}
\definecolor{dviolet}{HTML}{42208C}
\definecolor{labelkey}{HTML}{025951}
\definecolor{refkey}{HTML}{025951}
\definecolor{refkey}{rgb}{0,0.6,0.0}
\definecolor{Brown}{rgb}{0.45,0.0,0.05}
\definecolor{dgreen}{rgb}{0.00,0.49,0.00}
\definecolor{dblue}{rgb}{0,0.18,0.75}
\definecolor{lblue}{rgb}{0,0.7,0.75}
\definecolor{dviolet}{HTML}{9400D3}
\definecolor{pblue}{rgb}{0.1176,0.5647,1}
\definecolor{nblue}{rgb}{0.2,0.3,1}
\definecolor{pgreen}{rgb}{0.1961,0.8039,0.1961}
\definecolor{ngreen}{rgb}{0.0,0.6,0.3}
\definecolor{pred}{rgb}{1.0,0.2706,0.0}
\definecolor{magenta}{HTML}{ff00ff}
\definecolor{hotmagenta}{rgb}{1.0, 0.11, 0.81}
\definecolor{dorng}{rgb}{0.91,0.41,0.17}
\definecolor{dgray}{rgb}{0.41,0.41,0.41}
\usepackage{tikz,tkz-euclide,pgfplots}
\usetikzlibrary{arrows,decorations}
\addtokomafont{section}{\centering}

\usepackage{enumitem}
\setlist{itemsep=-2.0pt}

\usepackage{upref}  
\usepackage{hyperref}
\hypersetup{colorlinks=true,
linktocpage=true,
linkcolor=dblue,
citecolor=dgreen,
urlcolor=dred,
pdfencoding=auto,
hypertexnames=false}
\makeatletter
\g@addto@macro\th@plain{
\thm@headfont{\bfseries\sffamily}
\thm@notefont{}}
\g@addto@macro\th@definition{
\thm@headfont{\bfseries\sffamily}
\thm@notefont{}}
\g@addto@macro\th@remark{
\thm@headfont{\bfseries\sffamily}
\thm@notefont{}}
\makeatother
\theoremstyle{plain}
\newtheorem{theorem}{Theorem}[section]
\newtheorem{proposition}[theorem]{Proposition}
\newtheorem{corollary}[theorem]{Corollary}
\newtheorem{lemma}[theorem]{Lemma}

\theoremstyle{definition}

\newtheorem{problem}[theorem]{Problem}
\newtheorem{assumption}[theorem]{Assumption}
\theoremstyle{remark}

\newtheorem{algorithm}[theorem]{Algorithm}

\usepackage{epsfig}
\usepackage[usenames,dvipsnames]{pstricks}
\usepackage{pstricks-add}
\usepackage{pst-grad}
\usepackage{pst-plot} 
\usepackage{pst-node} 
\usepackage{pst-eucl} 
\usepackage{pst-coil} 
\usepackage[extdef=true]{delimset}
\DeclareMathDelimiterSet{\scal}[2]{
\selectdelim[l]<{#1}
\mathpunct{}\selectdelim[p]|
{#2}\selectdelim[r]>}
\DeclareMathDelimiterSet{\EC}[2]{
\mathsf{E}\selectdelim[l]({#1}
\mathpunct{}\selectdelim[p]|
{#2}\selectdelim[r])}

\newcommand{\menge}[2]{\bigl\{{#1}\mid{#2}\bigr\}} 
\DeclareMathDelimiterSet{\Menge}[2]{\selectdelim[l]\{
{#1}\selectdelim[m]|{#2}\selectdelim[r]\}}

\makeatletter
\def\upintkern@{\mkern-7mu\mathchoice{\mkern-3.5mu}{}{}{}}
\def\upintdots@{\mathchoice{\mkern-4mu\@cdots\mkern-4mu}%
{{\cdotp}\mkern1.5mu{\cdotp}\mkern1.5mu{\cdotp}}%
{{\cdotp}\mkern1mu{\cdotp}\mkern1mu{\cdotp}}%
{{\cdotp}\mkern1mu{\cdotp}\mkern1mu{\cdotp}}}
\makeatother
\DeclareFontFamily{OMX}{mdbch}{}
\DeclareFontShape{OMX}{mdbch}{m}{n}{ <->s * [0.8]  mdbchr7v }{}
\DeclareFontShape{OMX}{mdbch}{b}{n}{ <->s * [0.8]  mdbchb7v }{}
\DeclareFontShape{OMX}{mdbch}{bx}{n}{<->ssub * mdbch/b/n}{}
\DeclareSymbolFont{uplargesymbols}{OMX}{mdbch}{m}{n}
\SetSymbolFont{uplargesymbols}{bold}{OMX}{mdbch}{b}{n}
\DeclareMathSymbol{\upintop}{\mathop}{uplargesymbols}{82}
\DeclareMathSymbol{\upointop}{\mathop}{uplargesymbols}{"48}
\makeatletter
\renewcommand{\int}{\DOTSI\upintop\ilimits@}
\renewcommand{\oint}{\DOTSI\upointop\ilimits@}
\makeatother

\newcommand{\RR}{\mathbb{R}}
\newcommand{\NN}{\mathbb{N}}
\newcommand{\XX}{\EuScript{X}}

\newcommand{\CS}{\mathsf{C}}
\newcommand{\HS}{\mathsf{H}}

\newcommand{\ZS}{\mathsf{Z}}
\newcommand{\zS}{\mathsf{z}}
\newcommand{\nS}{{\mathsf{n}}}
\newcommand{\nnn}{\mathsf{n}\in\mathbb{N}}

\newcommand{\iS}{\mathsf{i}}
\newcommand{\dS}{\mathsf{d}}
\newcommand{\jS}{\mathsf{j}}
\newcommand{\kS}{\mathsf{k}}
\newcommand{\xS}{\mathsf{x}}
\newcommand{\yS}{\mathsf{y}}
\newcommand{\fS}{\mathsf{f}}
\newcommand{\gS}{\mathsf{g}}
\newcommand{\sS}{\mathsf{s}}
\newcommand{\lS}{\mathsf{l}}

\newcommand{\BE}{\EuScript{B}}
\newcommand{\FE}{\EuScript{F}}

\newcommand{\pinf}{{+}\infty}
\newcommand{\minf}{{-}\infty}
\newcommand{\zeroun}{\intv[o]{0}{1}}

\newcommand{\RX}{\intv[l]0{\minf}{\pinf}}
\newcommand{\RP}{\intv[r]0{0}{\pinf}}
\newcommand{\RPP}{\intv[o]0{0}{\pinf}}
\newcommand{\RPX}{\intv0{0}{\pinf}}

\newcommand{\emp}{\varnothing}

\newcommand{\Int}{\displaystyle\int}
 
\newcommand{\minimize}[2]{\underset{\substack{{#1}}}
{\operatorname{minimize}}\;\;#2}

\newcommand{\pushfwd}%
{\ensuremath{\mbox{\Large$\,\triangleright\,$}}}

\DeclareMathOperator{\Argmin}{Argmin}

\newcommand{\Id}{\mathsf{Id}}

\newcommand{\moyo}[2]{\leftindex[I]^{#2}{#1}}
\DeclareMathOperator{\card}{card}

\DeclareMathOperator{\dom}{dom}

\DeclareMathOperator{\prox}{prox}
\DeclareMathOperator{\proj}{proj}

\newcommand{\EE}{\mathsf{E}}
\newcommand{\PP}{\mathsf{P}}

\renewcommand{\leq}{\leqslant}
\renewcommand{\geq}{\geqslant}

\newcommand{\Pas}{\text{\normalfont$\PP$-a.s.}}

\renewenvironment{abstract}{%
\vspace*{-0.50cm}
\small
\quotation%
\noindent%
{\normalfont\bfseries\sffamily
\nobreak\abstractname\ }%
}{%
\endquotation%
\medskip
}
\renewcommand{\abstractname}{Abstract.}
\newcommand\keywordsname{Keywords.}
\newenvironment{keywords}
{\renewcommand\abstractname{\keywordsname}\begin{abstract}}
{\end{abstract}}
\usepackage[auth-sc]{authblk}
\newcommand{\email}[1]{\href{mailto:#1}{\nolinkurl{#1}}}
\renewcommand*\Affilfont{\normalfont\normalsize}
\newcommand\affilcr{\protect\\ \protect\Affilfont}
\makeatletter
\renewcommand\AB@affilsepx{\protect\\[0.5em]}
\makeatother

\author[1]{Javier I. Madariaga}
\affil[1]{North Carolina State University
\affilcr
Department of Mathematics
\affilcr
Raleigh, NC 27695, USA
\affilcr
\email{jimadari@ncsu.edu}
}

\begin{document}

\title{Convergence of the Iterates of the\\ 
Stochastic Proximal Gradient Method\thanks{Contact author:
J. I. Madariaga. Email: \email{jimadari@ncsu.edu}.
This work was supported by the National
Science Foundation under grant DMS-2513409.
}}

\date{~}

\maketitle

\begin{abstract}
We propose a novel study of the stochastic proximal gradient method
for minimizing the sum of two convex functions, one of which is
smooth. Under suitable assumptions and without requiring any
boundedness or control of the variance of the random variables, we
derive the almost sure convergence and the convergence in the mean
of the iterates to a solution of the minimization
problem. The results are applied to classification and convex
feasibility problems.
\end{abstract}


\newpage

\section{Introduction}
\label{sec:1}

Let $\HS$ be a Euclidean space and let $(\upOmega,\FE,\PP)$ be a
complete probability space. We consider the minimization of 
the sum of two functions in $\upGamma_{0}(\HS)$, one of which is
smooth. As shown in \cite{Acnu24}, this framework models many
problems in applied mathematics and engineering.

\begin{problem}
\label{prob:0}
Let $\upbeta\in\RPP$, let $\fS\in\upGamma_{0}(\HS)$, and let
$\gS\colon\HS\to\RR$ be convex, differentiable, and such that
$\nabla\gS$ is $\upbeta$-Lipschitzian, with
$\Argmin(\fS+\gS)\neq\emp$. The task is to
\begin{equation}
\label{e:prob0}
\minimize{\xS\in\HS}{\fS(\xS)+\gS(\xS)}.
\end{equation}
\end{problem}

A standard approach to solve this problem is to
use the forward-backward algorithm, also called the
proximal-gradient algorithm in this context: Given $\xS_{0}\in\HS$ 
and a sequence $(\upgamma_{\nS})_{\nnn}$ in 
$\RPP$ such that
$0<\inf\upgamma_{\nS}\leq\sup\upgamma_{\nS}<2/\upbeta$, iterate 
\begin{equation}
\label{e:pgm}
(\forall\nnn)\quad
\xS_{\nS+1}=\prox_{\upgamma_{\nS}\mathsf{f}}
\brk1{\xS_{\nS}-\upgamma_{\nS}\nabla\gS(\xS_{\nS})}.
\end{equation}
Then $(\xS_{\nS})_{\nnn}$ is guaranteed to converge to a
solution to Problem~\ref{prob:0} \cite{Smms05,Tsen91}. The implementation
of this
method requires both the proximity operator of $\fS$ and the
gradient of $\gS$ to be numerically tractable. However, evaluating
these operators could be computationally expensive or even
impossible. This paper investigates the following
version of Problem~\ref{prob:0}.

\begin{problem}
\label{prob:1}
In the context of Problem~\ref{prob:0}, let
$(\mathsf{K},\EuScript{K})$ be a measurable space. For every
$\kS\in\mathsf{K}$, let $\fS_{\kS}\in\upGamma_{0}(\HS)$ and let
$\gS_{\kS}\colon\HS\to\RR$ be convex and differentiable. Further,
let $k\colon(\upOmega,\FE,\PP)\to(\mathsf{K},\EuScript{K})$ be a
random variable such that 
\begin{equation}
\label{e:p1}
\begin{cases}
(\forall\xS\in\dom\fS)&\fS(\xS)
=\Int_{\upOmega}\fS_{k(\upomega)}(\xS)\PP(d\upomega);\\[3mm]
(\forall\xS\in\HS)&\gS(\xS)
=\Int_{\upOmega}\gS_{k(\upomega)}(\xS)\PP(d\upomega).
\end{cases}
\end{equation}
\end{problem}

The contribution of this paper is
to provide new results on the asymptotic behavior of the
following stochastic version of the proximal-gradient method for
solving Problem~\ref{prob:1}.

\begin{algorithm}
\label{algo:1}
In the setting of Problem~\ref{prob:1}, let
$(\upgamma_{\nS})_{\nnn}$ be a sequence in $\RPP$ and let
$x_{0}\in L^2(\upOmega,\FE,\PP;\HS)$. Iterate
\begin{equation}
\label{e:algo1}
\hskip -1mm 
\begin{array}{l}
\text{for}\;\nS=0,1,\ldots\\
\left\lfloor
\begin{array}{l}
k_{\nS}\;\textup{is a copy of $k$ and is independent of}\;
\upsigma(x_{0},\ldots,x_{\nS})\\
x_{\nS+1}=\prox_{\upgamma_{\nS}\mathsf{f}_{k_{\nS}^{}}}
\brk1{x_{\nS}-\upgamma_{\nS}\nabla\gS_{k_{\nS}}(x_{\nS})}.
\end{array}
\right.\\
\end{array}
\end{equation}
\end{algorithm}

Algorithm~\ref{algo:1} can be interpreted as the inexact version
\begin{equation}
(\forall\nnn)\quad 
x_{\nS+1}=\prox_{\upgamma_{\nS}\mathsf{f}}
\brk1{x_{\nS}-\upgamma_{\nS}\brk{\nabla\gS(x_{\nS})+b_{\nS}}}
+a_{\nS},
\end{equation}
of \eqref{e:pgm}, in which
\begin{equation}
(\forall\nnn)\quad
\begin{cases}
a_{\nS}&=\prox_{\upgamma_{\nS}\mathsf{f}_{k_{\nS}^{}}}
\brk1{x_{\nS}-\upgamma_{\nS}\nabla\gS_{k_{\nS}}(x_{\nS})}
-\prox_{\upgamma_{\nS}\mathsf{f}}
\brk1{x_{\nS}-\upgamma_{\nS}\nabla\gS_{k_{\nS}}(x_{\nS})};\\
b_{\nS}&=\nabla\gS_{k_{\nS}}(x_{\nS})-\nabla\gS(x_{\nS}).
\end{cases}
\end{equation}
Thus, we can establish the asymptotic behavior of
Algorithm~\ref{algo:1} through some stochastic inexact version of
the forward-backward algorithm; see, e.g., \cite{Sadd25,Comb16}.
However, the general frameworks established in these works do not
make use of the structure of the functions in \eqref{e:p1} and,
instead, they rely on strong conditions on the sequences
$(a_{\nS})_{\nnn}$ and $(b_{\nS})_{\nnn}$ such as convergence to
zero and boundedness of their variance. 

Algorithm~\ref{algo:1} has been studied in the case of a
deterministic function $\fS$, i.e., when only $\nabla\gS$ is
randomly approximated \cite{Atch17,Brav24,Cui22,Fort19,Rosa20,%
Xiao14}, and
in the case when $\gS=0$ \cite{Asi19,Eise22,Iidu17,Trao24}.
However, Algorithm~\ref{algo:1} in its full generality has been
less explored. Existing analyses either do not prove almost
sure convergence \cite{Herm23,Patr21} or rely on restrictive
assumptions, such as the uniform boundedness of all gradients and
subgradients \cite{Bert11}, or the existence of solutions with
subgradients in $L^{2\mathsf{p}}$ \cite{Bian16}. In the present 
paper, we establish almost-sure and $L^1$ convergence of the
sequence generated by Algorithm~\ref{algo:1} under much weaker
assumptions.
 
The rest of the paper is organized as follows. Section~\ref{sec:2}
introduces the general notation and the preliminary results used
throughout the manuscript. Section~\ref{sec:3} establishes the
convergence of Algorithm~\ref{algo:1} to a solution of
Problem~\ref{prob:1} under mild conditions. Section~\ref{sec:4}
proposes an application to mixed-loss classification problems
and Section~\ref{sec:5} to inconsistent convex feasibility
problems.

\section{Notation}
\label{sec:2}

Throughout, $\HS$ is a Euclidean space with identity operator
$\Id$, scalar product $\scal{\cdot}{\cdot}$, and associated norm
$\|\cdot\|$. Let
$\CS$ be a nonempty closed convex subset of $\HS$. Then
$\iota_{\CS}$ denotes the indicator function of $\CS$,
$\proj_{\CS}$ the projection operator onto $\CS$, 
$\mathsf{N}_{\CS}$ the normal cone to $\CS$,
and
$\dS_{\CS}\colon\xS\mapsto\inf_{\yS\in\CS}\norm{\xS-\yS}$ the
distance function of $\CS$. The class of lower semicontinuous
convex functions $\mathsf{h}\colon\HS\to\RX$ such that
$\dom\mathsf{h}=\menge{\xS\in\HS} {\mathsf{h}(\xS)<\pinf} \neq\emp$
is denoted by $\upGamma_{0}(\HS)$. Let
$\mathsf{h}\in\upGamma_{0}(\HS)$. The subdifferential of
$\mathsf{h}$ at $\xS\in\HS$ is the set
$\partial\mathsf{h}(\xS)=\menge{\mathsf{u}\in\HS\!}
{\!(\forall\zS\in\HS)\; \scal{\zS-\xS}{\mathsf{u}}+\mathsf{h}(\xS)
\!\leq\! \mathsf{h}(\zS)}$, the element of minimal norm in
$\partial\mathsf{h}(\xS)$ is $\moyo{\partial\mathsf{h}(\xS)}{0}$,
and the proximity operator of $\mathsf{h}$ is 
\begin{equation} 
\label{e:dprox}
\prox_{\mathsf{h}}\colon\HS\to\HS\colon
\xS\mapsto\underset{\zS\in\HS}
{\text{argmin}}\;\brk2{\mathsf{h}(\zS)+ \dfrac{1}{2}\|\xS-\zS\|^2}.
\end{equation}

Let $(\upXi,\EuScript{G})$ be a measurable space. A $\upXi$-valued
random variable is a measurable mapping
$x\colon(\upOmega,\FE,\PP)\to(\upXi,\EuScript{G})$. 
Given $x\colon\upOmega\to\upXi$ and $\mathsf{S}\in\EuScript{G}$, we
set $[x\in\mathsf{S}]=\menge{\upomega\in\upOmega}
{x(\upomega)\in\mathsf{S}}$. Let $x$ and $y$ be random variables 
from $(\upOmega,\FE,\PP)$ to $(\upXi,\EuScript{G})$. Then $y$ is a
copy of $x$ if, for every $\mathsf{S}\in\EuScript{G}$, 
$\PP([x\in\mathsf{S}])=\PP([y\in\mathsf{S}])$.
The Borel $\upsigma$-algebra of $\HS$ is denoted by $\BE$.
An $\HS$-valued random variable is a measurable 
mapping $x\colon(\upOmega,\FE)\to(\HS,\BE)$. 
Let $\mathsf{p}\in\left[1,\pinf\right[$. Then
$L^\mathsf{p}(\upOmega,\FE,\PP;\HS)$ denotes the space of 
equivalence classes of $\Pas$ equal $\HS$-valued random variables
$x\colon(\upOmega,\FE,\PP)\to(\HS,\BE)$ such that 
$\EE\|x\|^\mathsf{p}<\pinf$. Endowed with the norm 
\begin{equation}
\|\cdot\|_{L^\mathsf{p}(\upOmega,\FE,\PP;\HS)}\colon
x\mapsto\EE^{1/\mathsf{p}}\|x\|^\mathsf{p}
=\brk3{\int_{\upOmega}\|x(\upomega)\|^{\mathsf{p}}
\PP(d\upomega)}^{1/\mathsf{p}},
\end{equation}
$L^\mathsf{p}(\upOmega,\FE,\PP;\HS)$ is a
real Banach space. Further,
\begin{equation}
(\forall\mathsf{S}\in\BE)\quad
L^\mathsf{p}(\upOmega,\FE,\PP;\mathsf{S})=\Menge1{x\in
L^\mathsf{p}(\upOmega,\FE,\PP;\HS)}{x\in\mathsf{S}\;\Pas}.
\end{equation}
The $\upsigma$-algebra generated by a family $\upPhi$ of random 
variables is denoted by $\upsigma(\upPhi)$. 

The reader is referred to \cite{Livre1} for background on convex
analysis, and to \cite{Shir16} for
background on probability. 

\section{Convergence analysis}
\label{sec:3}

We propose to study the convergence of Algorithm~\ref{algo:1} to 
solutions to Problem~\ref{prob:1} under the following assumptions.

\begin{assumption}
\label{as:1}
In the context of Problem~\ref{prob:1}, there exists
$\bar{\zS}\in\Argmin(\fS+\gS)$ and, for every
$\kS\in\mathsf{K}$,
$\sS_{\kS}\in\partial\fS_{\kS}(\bar{\zS})$, such that
\begin{equation}
\int_{\upOmega}\brk1{\sS_{k(\upomega)}
+\nabla\gS_{k(\upomega)}(\bar{\zS})}\PP(d\upomega)
={0}
\;\;\text{and}\;\;
\int_{\upOmega}\norm1{\sS_{k(\upomega)}
+\nabla\gS_{k(\upomega)}(\bar{\zS})}^2\PP(d\upomega)
<\pinf.
\end{equation}
\end{assumption}

\begin{assumption}
\label{as:2}
In the context of Problem~\ref{prob:1}, 
\begin{equation}
(\forall\kS\in\mathsf{K})\quad
\dom\partial\fS_{\kS}=\dom\partial\fS.
\end{equation}
In addition, there exists an increasing function
$\uppsi\colon\RP\to\RP$ such that
\begin{equation}
\label{e:as3}
(\forall\xS\in\dom\partial\fS)\quad
\int_{\upOmega}\norm1{\moyo{\partial\fS_{k(\upomega)}(\xS)}{0}
+\nabla\gS_{k(\upomega)}(\xS)}^2\PP(d\upomega)
\leq\uppsi\brk1{\norm{\xS}}.
\end{equation}
\end{assumption}

Assumption~\ref{as:1} is significantly weaker than the standard
assumptions in the literature, which typically require uniform
boundedness of all measurable selections of subgradients at every
solution to Problem~\ref{prob:1} \cite{Asi19,Bert11}, or the
existence of subgradients in $L^{2\mathsf{p}}$ \cite{Bian16}.
Assumption~\ref{as:2} ensures that the sequence generated by
Algorithm~\ref{algo:1} remains within $\dom\fS$. In addition, it
allows for arbitrary subgradient growth, since $\uppsi$ can be any
increasing function. Assumption~\ref{as:2} is weaker than those in
the literature, which require controlling every
measurable selection of subgradients \cite{Asi19} or restricting
the function $\uppsi$ to a particular form \cite{Bian16}.

We present two technical lemmas.

\begin{lemma}
\label{l:1}
Let $\boldsymbol{x}=(x_1,\ldots,x_{\mathsf{N}})$ be an
$\HS^\mathsf{N}$-valued random variable, let
$(\mathsf{K},\EuScript{K})$ be a measurable space, and suppose that
the random variable
$k\colon(\upOmega,\FE,\PP)\to(\mathsf{K},\EuScript{K})$ is
independent of $\upsigma(\boldsymbol{x})$. Let
$\mathsf{h}\colon(\mathsf{K}\times\HS,\EuScript{K}\otimes\BE)
\to\RP$ be measurable and define
$\upphi\colon\HS\to\RPX\colon\xS\mapsto\EE\mathsf{h}(k,\xS)$. 
Then, for $\PP$-almost every $\upomega'\in\upOmega$, 
\begin{equation}
\label{e:cex}
\EC1{\mathsf{h}(k,x_1)}{\upsigma(\boldsymbol{x})}(\upomega')
=\int_{\upOmega}\mathsf{h}\brk1{k(\upomega),x_1(\upomega')}
\PP(d\upomega)
=\upphi\brk1{x_1(\upomega')}.
\end{equation}
\end{lemma}
\begin{proof}
The proof is analogous to that of \cite[Lemma~2.8]{Moco26},
replacing Fubini's theorem with Tonelli's theorem. 
\end{proof}

The following fact is used in several papers without proof.

\begin{lemma}
\label{l:2}
In the context of Algorithm~\ref{algo:1}, suppose that
Assumptions~\ref{as:1} and \ref{as:2} are in force.
Let $\nnn\smallsetminus\{0\}$, let $\zS\in\dom\fS$, and set
$\XX_{\nS}=\upsigma(x_{0},\ldots,x_{\nS})$. Then, with probability 
$1$,
\begin{enumerate}
\item
\label{l:2i}
$\EC1{\mathsf{f}_{k_{\nS}}(\zS)}{\XX_{\nS}}=\mathsf{f}(\zS)$,\;\;%
and\;\;
$\EC1{\mathsf{g}_{k_{\nS}}(\zS)}{\XX_{\nS}}=\mathsf{g}(\zS)$.
\item
\label{l:2ii}
$\EC1{\mathsf{f}_{k_{\nS}}(x_{\nS})}{\XX_{\nS}}=\mathsf{f}(x_{\nS})
$\;\;and\;\;
$\EC1{\mathsf{g}_{k_{\nS}}(x_{\nS})}{\XX_{\nS}}
=\mathsf{g}(x_{\nS})$.
\item
\label{l:2iii}
$\EC1{\nabla\mathsf{g}_{k_{\nS}}(\zS)}{\XX_{\nS}}
=\nabla\mathsf{g}(\zS)$\;\;and\;\;
$\EC1{\nabla\mathsf{g}_{k_{\nS}}(x_{\nS})}{\XX_{\nS}}
=\nabla\mathsf{g}(x_{\nS})$.
\end{enumerate}
\end{lemma}
\begin{proof}
\ref{l:2i}: This follows from \eqref{e:p1} and the fact that
$k_{\nS}$ is a copy of $k$.

\ref{l:2ii}: We note that
\begin{equation}
\fS_{k_{\nS}}(x_{\nS})
=\fS_{k_{\nS}}(x_{\nS})-\fS_{k_{\nS}}(\bar{\zS})
+\scal{\bar{\zS}-x_{\nS}}{\sS_{k_{\nS}}}
+\fS_{k_{\nS}}(\bar{\zS})
-\scal{\bar{\zS}-x_{\nS}}{\sS_{k_{\nS}}}\;\;\Pas
\end{equation}
It follows from the definition of the subdifferential that
$\fS_{k_{\nS}}(x_{\nS})-\fS_{k_{\nS}}(\bar{\zS})
+\scal{\bar{\zS}-x_{\nS}}{\sS_{k_{\nS}}}\geq 0\;\Pas$
Therefore, by Lemma~\ref{l:1}, \ref{l:2ii}, and 
Assumption~\ref{as:1}, we get
\begin{equation}
\EC1{\fS_{k_{\nS}}(x_{\nS})}{\XX_{\nS}}
=\fS(x_{\nS})-\fS(\bar{\zS})
+\scal1{\bar{\zS}-x_{\nS}}{-\nabla\gS(\bar{\zS})}
+\fS(\bar{\zS})-\scal1{\bar{\zS}-x_{\nS}}{-\nabla\gS(\bar{\zS})}
=\fS(x_{\nS})\;\;\Pas
\end{equation}
Similarly, we deduce 
$\EC1{\mathsf{g}_{k_{\nS}}(x_{\nS})}{\XX_{\nS}}
=\mathsf{g}(x_{\nS})\;\Pas$

\ref{l:2iii}: From Assumption~\ref{as:2},
$\EE\norm{\nabla\gS_{k}(\cdot)}^2$ is locally bounded. Thus, the
conclusion follows from \ref{l:2ii} and the dominated convergence
theorem.
\end{proof}

We now show the almost sure convergence and $L^1$ of the iterates
of Algorithm~\ref{algo:1}.

\begin{theorem}
\label{t:1}
In the setting of Problem~\ref{prob:1}, suppose that
Assumptions~\ref{as:1} and \ref{as:2} are in force, and let 
$(x_{\nS})_{\nnn}$ be the sequence generated by
Algorithm~\ref{algo:1}. In addition,
suppose that $\sum_{\nnn}\upgamma_{\nS}=\pinf$ and
$\sum_{\nnn}\upgamma_{\nS}^2<\pinf$.
Then the following hold:
\begin{enumerate}
\item
\label{t:1i}
$(x_{\nS})_{\nnn}$ is bounded $\Pas$ in $\HS$ and bounded in
$L^2(\upOmega,\FE,\PP;\HS)$. 
\item
\label{t:1ii}
$\varliminf(\fS+\gS)(x_{\nS})=\inf(\fS+\gS)(\HS)\;\Pas$
\item
\label{t:1iii}
$(x_{\nS})_{\nnn}$ converges $\Pas$ to a random variable 
$x\in L^2(\upOmega,\FE,\PP;\Argmin(\fS+\gS))$.
\item
\label{t:1iv}
$\nabla\gS(x_{\nS})\to\nabla\gS(\bar{\zS})\;\Pas$
\item
\label{t:1v}
Suppose that there exists $\upxi\in\RPP$ such that 
$\uppsi=\upxi(1+\abs{\cdot}^2)$. Then $(x_{\nS})_{\nnn}$ converges
in $L^1(\upOmega,\FE,\PP;\HS)$ to $x$. 
\end{enumerate}
\end{theorem}
\begin{proof}
\ref{t:1i}: Let $\nnn$ and set
$\XX_{\nS}=\upsigma=(x_{0},\ldots,x_{\nS})$. It follows from
\cite[Proposition~12.28]{Livre1} that, for $\PP$-almost every
$\upomega\in\upOmega$,
$\prox_{\upgamma_{\nS}\fS_{k_{\nS}(\upomega)}}$ is firmly
nonexpansive, hence nonexpansive. On the other hand, we deduce from
Assumption~\ref{as:1} and the characterization of the proximity
operator \cite[Proposition~16.44]{Livre1} that, for $\PP$-almost
every $\upomega\in\upOmega$,
\begin{equation}
\label{e:103}
\sS_{k_{\nS}(\upomega)}^{}
\in\partial\fS_{k_{\nS}(\upomega)}(\bar{\zS})
\Leftrightarrow\bar{\zS}+\upgamma_{\nS}\sS_{k_{\nS}(\upomega)}^{}
-\bar{\zS}
\in\upgamma_{\nS}\partial\fS_{k_{\nS}(\upomega)}(\bar{\zS})
\Leftrightarrow\bar{\zS}
=\prox_{\upgamma_{\nS}\fS_{k_{\nS}(\upomega)}}
\brk1{\bar{\zS}+\upgamma_{\nS}\sS_{k_{\nS}(\upomega)}}
\end{equation}
Then it follows from \eqref{e:algo1}, \eqref{e:p1}, and
\eqref{e:103} that 
\begin{align}
&{\norm{x_{\nS+1}-\bar{\zS}}^2}\nonumber\\
&\;\;=\norm1{\prox_{\upgamma_{\nS}\mathsf{f}_{k_{\nS}^{}}}
\brk1{x_{\nS}-\upgamma_{\nS}\nabla\gS_{k_{\nS}}(x_{\nS})}
-\prox_{\upgamma_{\nS}\mathsf{f}_{k_{\nS}^{}}}
\brk1{\bar{\zS}+\upgamma_{\nS}\sS_{k_{\nS}}}}^2\nonumber\\
&\;\;\leq\norm1{x_{\nS}-\upgamma_{\nS}\nabla\gS_{k_{\nS}}(x_{\nS})
-\brk1{\bar{\zS}+\upgamma_{\nS}\sS_{k_{\nS}}}}^2\nonumber\\
&\;\;=\norm1{\brk{x_{\nS}-\bar{\zS}}
-\upgamma_{\nS}\brk1{\nabla\gS_{k_{\nS}}(x_{\nS})
+\sS_{k_{\nS}}}}^2\nonumber\\
&\;\;=\norm{x_{\nS}-\bar{\zS}}^2
-2\upgamma_{\nS}\scal1{x_{\nS}-\bar{\zS}}
{\nabla\gS_{k_{\nS}}(x_{\nS})
+\sS_{k_{\nS}}}
+\upgamma_{\nS}^2{\norm1{\sS_{k_{\nS}}+\nabla\gS_{k_{\nS}}(x_{\nS})
}^2}\nonumber\\
&\;\;\leq\norm{x_{\nS}-\bar{\zS}}^2
-2\upgamma_{\nS}\scal1{x_{\nS}-\bar{\zS}}
{\nabla\gS_{k_{\nS}}(x_{\nS})
+\sS_{k_{\nS}}}
+2\upgamma_{\nS}^2{\norm1{\nabla\gS_{k_{\nS}}(x_{\nS})
-\nabla\gS_{k_{\nS}}(\bar{\zS})}^2}
+2\upgamma_{\nS}^2{\norm1{\sS_{k_{\nS}}
+\nabla\gS_{k_{\nS}}(\bar{\zS})}^2}\nonumber\\
&\;\;\leq\norm{x_{\nS}-\bar{\zS}}^2
-2\upgamma_{\nS}\scal1{x_{\nS}-\bar{\zS}}
{\nabla\gS_{k_{\nS}}(x_{\nS})
+\sS_{k_{\nS}}}
+2\upbeta\upgamma_{\nS}^2{\norm{x_{\nS}-\bar{\zS}}^2}
+2\upgamma_{\nS}^2{\norm1{\sS_{k_{\nS}}
+\nabla\gS_{k_{\nS}}(\bar{\zS})}^2}\nonumber\\
&\;\;=\brk1{1+2\upbeta\upgamma_{\nS}^2}
\norm{x_{\nS}-\bar{\zS}}^2
-2\upgamma_{\nS}\scal1{x_{\nS}-\bar{\zS}}
{\nabla\gS_{k_{\nS}}(x_{\nS})
+\sS_{k_{\nS}}}
+2\upgamma_{\nS}^2{\norm1{\sS_{k_{\nS}}
+\nabla\gS_{k_{\nS}}(\bar{\zS})}^2}\;\;\Pas
\label{e:104}
\end{align}
Therefore, since $k_{\nS}$ is independent of $\XX_{\nS}$, 
$x_{\nS}-\bar{\zS}$ is $\XX_{\nS}$-measurable, $\nabla\gS$
is $(1/\upbeta)$-cocoercive \cite[Corollary~18.17]{Livre1}, 
and by Lemmas~\ref{l:1} and \ref{l:2}, we get 
\begin{align}
&\EC1{\norm{x_{\nS+1}-\bar{\zS}}^2}{\XX_{\nS}}
\nonumber\\
&\;\;\leq\brk1{1+2\upbeta\upgamma_{\nS}^2}
\norm{x_{\nS}-\bar{\zS}}^2
-2\upgamma_{\nS}\scal2{x_{\nS}-\bar{\zS}}
{\EC1{\nabla\gS_{k_{\nS}}(x_{\nS})
+\sS_{k_{\nS}}}{\XX_{\nS}}}
+2\upgamma_{\nS}^2
\EC1{\norm{\sS_{k_{\nS}}+\nabla\gS_{k_{\nS}}(\bar{\zS})}^2}
{\XX_{\nS}}\nonumber\\
&\;\;=\brk1{1+2\upbeta\upgamma_{\nS}^2}
{\norm{x_{\nS}-\bar{\zS}}^2}
-2\upgamma_{\nS}
\scal1{x_{\nS}-\bar{\zS}}
{\EE{\nabla\gS_{k}(x_{\nS})}+\EE{\sS_{k}}}
+2\upgamma_{\nS}^2\EE{\norm{\sS_{k}
+\nabla\gS_{k}(\bar{\zS})}^2}
\nonumber\\
&\;\;=\brk1{1+2\upbeta\upgamma_{\nS}^2}
{\norm{x_{\nS}-\bar{\zS}}^2}
-2\upgamma_{\nS}
\scal1{x_{\nS}-\bar{\zS}}{\nabla\gS(x_{\nS})-\nabla\gS(\bar{\zS})}
+2\upgamma_{\nS}^2\EE{\norm{\sS_{k}
+\nabla\gS_{k}(\bar{\zS})}^2}
\nonumber\\
&\;\;\leq\brk1{1+2\upbeta\upgamma_{\nS}^2}
{\norm{x_{\nS}-\bar{\zS}}^2}
-2\upbeta^{-1}\upgamma_{\nS}
\norm1{\nabla\gS(x_{\nS})-\nabla\gS(\bar{\zS})}^2
+2\upgamma_{\nS}^2\EE{\norm{\sS_{k}
+\nabla\gS_{k}(\bar{\zS})}^2}
\;\;\Pas
\label{e:105}
\end{align}
Since $\sum_{\nnn}\upgamma_{\nS}^2<\pinf$ and
$\EE\norm{\sS_{k}+\nabla\gS_{k}(\bar{\zS})}^2<\pinf$, we deduce
from \eqref{e:105} that $(x_{\nS})_{\nnn}$ is stochastic
quasi-Fej\'erian relative to the set $\{\bar{\zS}\}$ in the sense
of \cite[Proposition~2.3]{Siop15}. It then follows from
\cite[Proposition~2.3(i) and (ii)]{Siop15} that $(x_{\nS})_{\nnn}$
is bounded $\Pas$ and
\begin{equation}
\label{e:106}
\sum_{\nnn}\upgamma_{\nS}
\norm1{\nabla\gS(x_{\nS})-\nabla\gS(\bar{\zS})}^2\leq\pinf\;\;\Pas
\end{equation}
Hence, the assumption $\sum_{\nnn}\upgamma_{\nS}=\pinf$ yields
\begin{equation}
\label{e:106+}
\varliminf\:\norm1{\nabla\gS(x_{\nS})-\nabla\gS(\bar{\zS})}^2=0\;\;
\Pas
\end{equation}
Similarly, by taking the expected value in \eqref{e:105} we get
\begin{equation}
(\forall\nnn)\quad\EE{\norm{x_{\nS+1}-\bar{\zS}}^2}
\leq\brk1{1+2\upgamma_{\nS}^2\upbeta}
\EE{\norm{x_{\nS}-\bar{\zS}}^2}
+2\upgamma_{\nS}^2\EE{\norm{\sS_{k}
+\nabla\gS_{k}(\bar{\zS})}^2},
\label{e:107}
\end{equation}
which shows that $(x_{\nS})_{\nnn}$ is quasi-Fej\'erian in
$L^2(\upOmega,\FE,\PP;\HS)$ relative to the set $\{\bar{\zS}\}$
\cite[Definition~3.1]{Nonl13}. Hence, it follows from
\cite[Proposition~3.2(ii)]{Nonl13} that $(x_{\nS})_{\nnn}$ is
bounded in $L^2(\upOmega,\FE,\PP;\HS)$.

\ref{t:1ii}: Denote $\upvarphi=\fS+\gS$ and, for every
$\kS\in\mathsf{K}$, $\upvarphi_{\kS}=\fS_{\kS}+\gS_{\kS}$. Let
$\zS\in\Argmin\upvarphi$, and let $\nnn\smallsetminus\{0\}$. We
infer from Assumption~\ref{as:2} and \eqref{e:algo1} that
$x_{\nS}\in\dom\partial\fS_{k_{\nS-1}}\subset\dom\fS\;\Pas$ and,
similarly to \eqref{e:103} and \eqref{e:104}, we deduce that
\begin{align}
\norm{x_{\nS+1}-x_{\nS}}
&=\norm1{\prox_{\upgamma_{\nS}\mathsf{f}_{k_{\nS}^{}}}
\brk1{x_{\nS}-\upgamma_{\nS}\nabla\gS_{k_{\nS}}(x_{\nS})}
-\prox_{\upgamma_{\nS}\mathsf{f}_{k_{\nS}^{}}}
\brk1{x_{\nS}
+\upgamma_{\nS}\moyo{\partial\fS_{k_{\nS}}}{0}(x_{\nS})}}
\nonumber\\
&\leq\norm1{\upgamma_{\nS}\moyo{\partial\fS_{k_{\nS}}}{0}(x_{\nS})
+\upgamma_{\nS}\nabla\gS_{k_{\nS}}(x_{\nS})}\nonumber\\
&=\upgamma_{\nS}\norm1{\moyo{\partial\fS_{k_{\nS}}}{0}(x_{\nS})
+\nabla\gS_{k_{\nS}}(x_{\nS})}\;\;\Pas
\label{e:108}
\end{align}
On the other hand, it follows from \eqref{e:algo1} and
\cite[Proposition~12.26 and Theorem~18.15]{Livre1} that
\begin{equation}
\begin{cases}
\upgamma_{\nS}\fS_{k_{\nS}}(x_{\nS+1})
\leq\upgamma_{\nS}\fS_{k_{\nS}}(\zS)
+\scal1{x_{\nS+1}-\zS}{x_{\nS}-x_{\nS+1}
-\upgamma_{\nS}\nabla\gS_{k_{\nS}}(x_{\nS})}\;\;\Pas\\[1mm]
\upgamma_{\nS}\gS_{k_{\nS}}(x_{\nS+1})
\leq\upgamma_{\nS}\gS_{k_{\nS}}(\zS)
+\scal1{x_{\nS+1}-\zS}{\upgamma_{\nS}\nabla\gS_{k_{\nS}}(x_{\nS})}
+\dfrac{\upgamma_{\nS}\upbeta}{2}\norm{x_{\nS+1}-x_{\nS}}^2\;\;\Pas
\end{cases}
\end{equation}
Thus, after adding both inequalities and rearranging the terms, we
obtain
\begin{equation}
\label{e:108+}
\scal1{x_{\nS+1}-\zS}{x_{\nS+1}-x_{\nS}}
\leq\upgamma_{\nS}\brk1{\upvarphi_{k_{\nS}}(\zS)
-\upvarphi_{k_{\nS}}(x_{\nS+1})}
+\dfrac{\upgamma_{\nS}\upbeta}{2}\norm{x_{\nS+1}-x_{\nS}}^2\;\;\Pas
\end{equation}
Then \eqref{e:108+}, the definition of the subdifferential, and 
\eqref{e:108} yield 
\begin{align}
&\scal1{x_{\nS+1}-\zS}{x_{\nS+1}-x_{\nS}}\nonumber\\
&\quad\leq\upgamma_{\nS}\brk1{\upvarphi_{k_{\nS}}(\zS)
-\upvarphi_{k_{\nS}}(x_{\nS})+\upvarphi_{k_{\nS}}(x_{\nS})
-\upvarphi_{k_{\nS}}(x_{\nS+1})} 
+\dfrac{\upgamma_{\nS}\upbeta}{2}\norm{x_{\nS+1}-x_{\nS}}^2
\nonumber\\
&\quad\leq\upgamma_{\nS}\brk2{\upvarphi_{k_{\nS}}(\zS)
-\upvarphi_{k_{\nS}}(x_{\nS})
+\scal1{x_{\nS}-x_{\nS+1}}{\moyo{\partial\fS_{k_{\nS}}}{0}(x_{\nS})
+\nabla\gS_{k_{\nS}}(x_{\nS})}} 
+\dfrac{\upgamma_{\nS}\upbeta}{2}\norm{x_{\nS+1}-x_{\nS}}^2
\nonumber\\
&\quad\leq\upgamma_{\nS}\brk1{\upvarphi_{k_{\nS}}(\zS)
-\upvarphi_{k_{\nS}}(x_{\nS})}
+\upgamma_{\nS}\norm{x_{\nS}-x_{\nS+1}}
\norm1{\moyo{\partial\fS_{k_{\nS}}}{0}(x_{\nS})
+\nabla\gS_{k_{\nS}}(x_{\nS})} 
+\dfrac{\upgamma_{\nS}\upbeta}{2}\norm{x_{\nS+1}-x_{\nS}}^2
\nonumber\\
&\quad\leq\upgamma_{\nS}\brk1{\upvarphi_{k_{\nS}}(\zS)
-\upvarphi_{k_{\nS}}(x_{\nS})}
+\upgamma_{\nS}^2\norm1{\moyo{\partial\fS_{k_{\nS}}}{0}(x_{\nS})
+\nabla\gS_{k_{\nS}}(x_{\nS})}^2 
+\dfrac{\upgamma_{\nS}\upbeta}{2}\norm{x_{\nS+1}-x_{\nS}}^2\;\;\Pas
\label{e:109}
\end{align}
Thus, we deduce from \eqref{e:109}, Lemmas~\ref{l:1} and \ref{l:2},
and Assumption~\ref{as:2} that, with probability~$1$, 
\begin{align}
&\EC1{\norm{x_{\nS+1}-\zS}^2}{\XX_{\nS}}
\nonumber\\
&\;\;=\norm{x_{\nS}-\zS}^2
+2\EC1{\scal{x_{\nS}-\zS}{x_{\nS+1}-x_{\nS}}}{\XX_{\nS}}
+\EC1{\norm{x_{\nS+1}-x_{\nS}}^2}{\XX_{\nS}}
\nonumber\\
&\;\;=\norm{x_{\nS}-\zS}^2
+2\EC1{\scal{x_{\nS+1}-\zS}{x_{\nS+1}-x_{\nS}}}{\XX_{\nS}}
-\EC1{\norm{x_{\nS+1}-x_{\nS}}^2}{\XX_{\nS}}
\nonumber\\
&\;\;\leq\norm{x_{\nS}-\zS}^2
+2\upgamma_{\nS}\EC1{\upvarphi_{k_{\nS}}(\zS)
-\upvarphi_{k_{\nS}}(x_{\nS})}{\XX_{\nS}}+2\upgamma_{\nS}^2\EC1{
\norm1{\moyo{\partial\fS_{k_{\nS}}}{0}(x_{\nS})
+\nabla\gS_{k_{\nS}}(x_{\nS})}^2}{\XX_{\nS}}
\nonumber\\
&\qquad+{\upgamma_{\nS}\upbeta}
\EC1{\norm{x_{\nS+1}-x_{\nS}}^2}{\XX_{\nS}}
-\EC1{\norm{x_{\nS+1}-x_{\nS}}^2}{\XX_{\nS}}
\nonumber\\
&\;\;=\norm{x_{\nS}-\zS}^2
+2\upgamma_{\nS}\brk1{\upvarphi(\zS)
-\upvarphi(x_{\nS})}+2\upgamma_{\nS}^2\EC1{
\norm1{\moyo{\partial\fS_{k_{\nS}}}{0}(x_{\nS})
+\nabla\gS_{k_{\nS}}(x_{\nS})}^2}{\XX_{\nS}}
\nonumber\\
&\qquad+\brk1{{\upgamma_{\nS}\upbeta}-1}
\EC1{\norm{x_{\nS+1}-x_{\nS}}^2}{\XX_{\nS}}
\nonumber\\
&\;\;\leq\norm{x_{\nS}-\zS}^2
+2\upgamma_{\nS}\brk1{\upvarphi(\zS)-\upvarphi(x_{\nS})}
+2\upgamma_{\nS}^2\uppsi\brk1{\norm{x_{\nS}}}
+\brk1{\upgamma_{\nS}\upbeta-1}
\EC1{\norm{x_{\nS+1}-x_{\nS}}^2}{\XX_{\nS}}\nonumber\\
&\;\;\leq\norm{x_{\nS}-\zS}^2
+2\upgamma_{\nS}\brk1{\upvarphi(\zS)-\upvarphi(x_{\nS})}
+2\upgamma_{\nS}^2\uppsi\brk1{\norm{x_{\nS}}}
+\max\brk[c]1{0,\upgamma_{\nS}\upbeta-1}
\EC1{\norm{x_{\nS+1}-x_{\nS}}^2}{\XX_{\nS}}.
\label{e:110}
\end{align}
We infer from \ref{t:1i} and Assumption~\ref{as:2} that
$\uppsi(\norm{x_{\nS}})$ is bounded $\Pas$ In addition,
$\sum_{\nnn}\upgamma_{\nS}^2<\pinf$ yields
$\upgamma_{\nS}\upbeta-1<0$ for $\nS$ large enough. Altogether,
\eqref{e:110} yields that $(x_{\nS})_{\nnn}$ is stochastic
quasi-Fej\'erian relative to $\Argmin\upvarphi$ in the sense of
\cite[Proposition~2.3]{Siop15}. It then follows from
\cite[Proposition~2.3(ii)]{Siop15} that 
\begin{equation}
\label{e:111}
\sum_{\nnn}\upgamma_{\nS}
\brk1{\upvarphi(x_{\nS})-\upvarphi(\zS)}\leq\pinf\;\;\Pas
\end{equation}
Hence, since $\sum_{\nnn}\upgamma_{\nS}=\pinf$, we have 
$\varliminf\:\upvarphi(x_{\nS})=\inf\upvarphi(\HS)\;\Pas$

\ref{t:1iii}: Let us show that $(x_{\nS})_{\nnn}$ corresponds to a
sequence generated by \cite[Algorithm~3.4]{Moco26}. To this end,
set $\ZS=\Argmin\upvarphi$ and
\begin{equation}
(\forall\nnn)\quad
\begin{cases}
t_{\nS}^*=2\brk{x_{\nS}-x_{\nS+1}}\in L^2(\upOmega,\FE,\PP;\HS);\\
\eta_{\nS}=\scal1{x_{\nS+1}+x_{\nS}}{x_{\nS}-x_{\nS+1}}
\in L^1(\upOmega,\FE,\PP;\RR);\\
\alpha_{\nS}=1;\\
\varepsilon_{\nS}=2\upgamma_{\nS}^2\uppsi\brk1{\norm{x_{\nS}}}
+\max\brk[c]1{0,\upgamma_{\nS}\upbeta-1}
\EC1{\norm{x_{\nS+1}-x_{\nS}}^2}{\XX_{\nS}}\in\RP\;\;\Pas;\\
\lambda_{\nS}=\dfrac{1}{2}.
\end{cases}
\end{equation}
The Cauchy--Schwarz inequality shows that 
\begin{equation}
(\forall\nnn)\quad
\dfrac{\mathsf{1}_{[t_{\nS}^*\neq0]}\mathsf{1}_{
\left[\scal{x_{\nS}}{t_{\nS}^*}>\eta_{\nS}\right]}\eta_{\nS}}
{\|t_{\nS}^*\|+\mathsf{1}_{[t_{\nS}^*=0]}}\in
L^2(\upOmega,\FE,\PP;\RR).
\end{equation}
In addition, we can show analogously to \eqref{e:109} and
\eqref{e:110} that, for every $\nnn$ and every $\zS\in\ZS$,
\begin{equation}
\scal1{\zS}{\EC{\alpha_{\nS}t^*_{\nS}}{\XX_{\nS}}}\leq
\EC{\alpha_{\nS}\eta_{\nS}}{\XX_{\nS}}+
\varepsilon_{\nS}\;\;\Pas
\end{equation}
Finally, we derive that
\begin{equation}
(\forall\nnn)\quad
x_{\nS+1}=x_{\nS}-\lambda_{\nS}\alpha_{\nS}t_{\nS}^*.
\end{equation}
Altogether, we confirm that $(x_{\nS})_{\nnn}$ is a sequence
constructed by \cite[Algorithm~3.4]{Moco26}. On the other hand, it
follows from \ref{t:1i} and \ref{t:1ii} that there exists
$\upOmega'\in\FE$ such that
\begin{equation}
\PP(\upOmega')=1\;\text{and}\;(\forall\upomega\in\upOmega')\;\;
\brk[s]1{({x_{\nS}(\upomega)})_{\nnn}\;\text{is bounded and}\;
\varliminf\upvarphi(x_{\nS}(\upomega))=\inf\upvarphi(\HS)}.
\end{equation} 
Let $\upomega\in\upOmega'$ and let $(\jS_{\nS})_{\nnn}$ be a
strictly increasing sequence in $\NN$ such that
$\lim\upvarphi(x_{\jS_{\nS}}(\upomega))=\inf\upvarphi(\HS)$. Since
the sequence is bounded, there exists a point $\xS\in\HS$ and a 
further subsequence, say $x_{\lS_{\jS_{\nS}}}(\upomega)$, such that
$x_{\lS_{\jS_{\nS}}}(\upomega)\to\xS$. 
Note that the lower semicontinuity of $\fS$ and the continuity of
$\gS$ lead to the lower semicontinuity of $\upvarphi$, which
yields 
\begin{equation}
\inf\upvarphi(\HS)\leq\upvarphi(\xS)
\leq\lim\upvarphi\brk1{x_{\lS_{\jS_{\nS}}}(\upomega)}
=\inf\upvarphi(\HS).
\end{equation}
Hence $\xS\in\Argmin\upvarphi$. Since $\upomega$ is arbitrarily
taken in $\upOmega'$, we deduce that, for $\PP$-almost every
$\upomega\in\upOmega$, there exists a cluster point of
$(x_{\nS}(\upomega))_{\nnn}$ that belongs to $\Argmin\upvarphi$. 
Therefore, it follows from \cite[Theorem~3.6(i)(d)]{Moco26} and the
fact that $\sum_{\nnn}\varepsilon_{\nS}<\pinf\;\Pas$ that there
exists an $(\Argmin\upvarphi)$-valued random variable $x$ such that
$x_{\nS}\to x\;\Pas$ Furthermore, \ref{t:1i} and Fatou's lemma
guarantee that
\begin{equation}
0\leq\EE\norm{x}^2
\leq\EE\brk1{\varliminf\norm{x_{\nS}}^2}
\leq\varliminf\EE\norm{x_{\nS}}^2
\leq\sup\EE\norm{x_{\nS}}^2<\pinf,
\end{equation}
which shows that $x\in L^2(\upOmega,\FE,\PP;\Argmin\upvarphi)$. 

\ref{t:1iv}: The continuity of $\nabla\gS$ and \ref{t:1iii} yield
$\nabla\gS(x_{\nS})\to\nabla\gS(x)\;\Pas$ On the other hand,
\eqref{e:106+} shows that
$\varliminf\:\norm{\nabla\gS(x_{\nS})-\nabla\gS(\bar{\zS})}
=0\;\Pas$ Then
\begin{equation}
\norm1{\nabla\gS(x)-\nabla\gS(\bar{\zS})}
\leq\varliminf\brk2{\norm1{\nabla\gS(x_{\nS})-\nabla\gS(x)}
+\norm1{\nabla\gS(x_{\nS})-\nabla\gS(\bar{\zS})}}=0\;\;\Pas,
\end{equation}
which shows that $\nabla\gS(x)=\nabla\gS(\bar{\zS})\;\Pas$
Therefore, $\nabla\gS(x_{\nS})\to\nabla\gS(\bar{\zS})\;\Pas$

\ref{t:1v}: It follows from \ref{t:1i} that 
\begin{equation}
\sup_{\nnn}\EE\uppsi(\norm{x_{\nS}})
=\sup_{\nnn}\upxi\brk1{1+\EE\norm{x_{\nS}}^2}<\pinf.
\end{equation}
Hence $\sum_{\nnn}\EE\varepsilon_{\nS}<\pinf$ and the convergence
of $(x_{\nS})_{\nnn}$ to $x$ in $L^1(\upOmega,\FE,\PP;\HS)$ follows
from \cite[Theorem~3.6(ii)(d)]{Moco26}.
\end{proof}

We present two corollaries of Theorem~\ref{t:1} that introduce
novel almost surely convergent results for the stochastic proximal
point algorithm and the stochastic gradient method.

\begin{corollary}
\label{c:1}
Let $\fS\in\upGamma_{0}(\HS)$ and let $(\mathsf{K},\EuScript{K})$
be a measurable space. For every $\kS\in\mathsf{K}$, let
$\fS_{\kS}\in\upGamma_{0}(\HS)$ such that
$\dom\partial\fS_{\kS}=\dom\partial\fS$. Further, let
$k\colon(\upOmega,\FE,\PP)\to(\mathsf{K},\EuScript{K})$ be a random
variable such that 
\begin{equation}
(\forall\xS\in\dom\partial\fS)\quad\fS(\xS)
=\Int_{\upOmega}\fS_{k(\upomega)}(\xS)\PP(d\upomega).
\end{equation}
Assume that there exists $\bar{\zS}\in\Argmin\fS$ such that, for
every $\kS\in\mathsf{K}$, $s_{\kS}\in\partial\fS_{\kS}(\bar{\zS})$,
and
\begin{equation}
\int_{\upOmega}\sS_{k(\upomega)}\PP(d\upomega)={0}
\;\;\text{and}\;\;
\int_{\upOmega}\norm{\sS_{k(\upomega)}}^2\PP(d\upomega)<\pinf.
\end{equation}
Further, assume that there exists an increasing function
$\uppsi\colon\RP\to\RP$ such that
\begin{equation}
(\forall\xS\in\HS)\quad
\int_{\upOmega}\norm1{\moyo{\partial\fS_{k(\upomega)}(\xS)}{0}
}^2\PP(d\upomega)
\leq\uppsi\brk1{\norm{\xS}}.
\end{equation}
Let $x_{0}\in L^2(\upOmega,\FE,\PP;\HS)$ and let
$(\upgamma_{\nS})_{\nnn}$ be a sequence in $\RPP$ such that
$\sum_{\nnn}\upgamma_{\nS}=\pinf$ and
$\sum_{\nnn}\upgamma_{\nS}^2<\pinf$. Iterate
\begin{equation}
(\forall\nnn)\quad
x_{\nS+1}=\prox_{\upgamma_{\nS}\fS_{k_{\nS}}}(x_{\nS}).
\end{equation}
Then $(x_{\nS})_{\nnn}$ converges $\Pas$ to a random variable $x\in
L^2(\upOmega,\FE,\PP;\Argmin\fS)$.

\end{corollary}

\begin{corollary}
\label{c:2}
Let $\upbeta\in\RPP$ and let $\gS\colon\HS\to\RR$ be a convex
differentiable function such that $\nabla\gS$ is
$\upbeta$-Lipschitzian. Let $(\mathsf{K},\EuScript{K})$ be a
measurable space. For every $\kS\in\mathsf{K}$, let
$\gS_{\kS}\colon\HS\to\RR$ be a convex differentiable function such
that $\nabla\gS_{\kS}$ is $\upbeta$-Lipschitzian. Further, let
$k\colon(\upOmega,\FE,\PP)\to(\mathsf{K},\EuScript{K})$ be a random
variable such that 
\begin{equation}
(\forall\xS\in\HS)\quad\gS(\xS)
=\Int_{\upOmega}\gS_{k(\upomega)}(\xS)\PP(d\upomega).
\end{equation}
Assume that there exists an increasing function
$\uppsi\colon\RP\to\RP$ such that
\begin{equation}
(\forall\xS\in\HS)\quad
\int_{\upOmega}\norm1{\nabla\gS_{k(\upomega)}(\xS)}^2\PP(d\upomega)
\leq\uppsi\brk1{\norm{\xS}}.
\end{equation}
Let $x_{0}\in L^2(\upOmega,\FE,\PP;\HS)$ and let
$(\upgamma_{\nS})_{\nnn}$ be a sequence in $\RPP$ such that
$\sum_{\nnn}\upgamma_{\nS}=\pinf$ and
$\sum_{\nnn}\upgamma_{\nS}^2<\pinf$. Iterate
\begin{equation}
(\forall\nnn)\quad
x_{\nS+1}=x_{\nS}-\upgamma_{\nS}\nabla\gS_{k_{\nS}}(x_{\nS}).
\end{equation}
Then $(x_{\nS})_{\nnn}$ converges $\Pas$ to a random variable $x\in
L^2(\upOmega,\FE,\PP;\Argmin\gS)$.
\end{corollary}

\section{Application to mixed-loss classification problems}
\label{sec:4}

We address a binary classification problem which is
modeled via the combination of two loss functions.

\begin{problem}
\label{prob:5}
The training data samples are split into two finite collections in 
$\RR^{\mathsf{N}}\times\{-1,1\}$:
$(\mathsf{u}_{\kS},\upxi_{\kS})_{\kS\in\mathsf{K}_1}$ and 
$(\mathsf{u}_{\kS},\upxi_{\kS})_{\kS\in\mathsf{K}_2}$.
Let $\upalpha\in\zeroun$. The task is to
\begin{equation}
\label{e:HL}
\minimize{\xS\in\RR^\mathsf{N}}
{\frac{1}{\card{\mathsf{K}_1}}\sum_{\kS\in\mathsf{K}_1}
\fS_{\kS}(\xS)
+\frac{1}{\card{\mathsf{K}_2}}\sum_{\kS\in\mathsf{K}_2}
\gS_{\kS}(\xS)},
\end{equation}
where
\begin{equation}
\label{e:p51}
\begin{cases}
(\forall\kS\in\mathsf{K}_1)\quad
\fS_{\kS}(\xS)=\upalpha\max\{0,1-\upxi_{\kS}
\scal{\xS}{\mathsf{u}_{\kS}}\};\\
(\forall\kS\in\mathsf{K}_2)\quad
\gS_{\kS}(\xS)=
(1-\upalpha)
\ln\brk1{1+\exp\brk1{-\upxi_{\kS}\scal{\xS}{\mathsf{u}_{\kS}}}}.
\end{cases}
\end{equation}
\end{problem}

Mixed-loss problems, in particular Problem~\ref{prob:5}, are
commonly used to train multi-task learning models; see, e.g.,
\cite{Chen24}. The goal of Problem~\ref{prob:5} is to learn a
linear classifier $\xS\in\HS=\RR^{\mathsf{N}}$ by minimizing the
mixed-loss function. In our model, $\mathsf{K}_1$ represents a set
of noisy data, which we handle using the hinge loss, whereas
$\mathsf{K}_2$ represents accurate data, for which we use the
logistic loss. To solve Problem~\ref{prob:5} using
Algorithm~\ref{algo:1}, let us first provide the proximity operator
of the hinge loss and the gradient of the logistic loss. As
shown in \cite[Example~24.37]{Livre1}, for every $\xS\in\HS$,
$\upgamma\in\RPP$, $\iS\in\mathsf{K}_1$, and $\jS\in\mathsf{K}_2$,
\begin{equation}
\label{e:ap6}
\begin{cases}
\prox_{\upgamma\fS_{\iS}}(\xS)
=
\begin{cases}
\xS,&\text{if}\;\upxi_{\iS}\scal{\mathsf{u}_{\iS}}{\xS}>1;\\
\xS+\dfrac{1-\upxi_{\iS}\scal{\mathsf{u}_{\iS}}{\xS}}
{\norm{\mathsf{u}_{\iS}}^2},
&\text{if}\;1\geq\upxi_{\iS}\scal{\mathsf{u}_{\iS}}{\xS}\geq
1-\upalpha\upgamma\norm{\mathsf{u}_{\iS}}^2;\\
\xS+\upalpha\upgamma\upxi_{\iS}\scal{\mathsf{u}_{\iS}}{\xS},
&\text{if}\;1-\upalpha\upgamma\norm{\mathsf{u}_{\iS}}^2>
\upxi_{\iS}\scal{\mathsf{u}_{\iS}}{\xS};
\end{cases}\\[1cm]
\nabla\gS_{\jS}(\xS)
=-\dfrac{(1-\upalpha)}
{1+\exp\brk1{\upxi_{\jS}\scal{\xS}{\mathsf{u}_{\jS}}}}
\upxi_{\jS}\mathsf{u}_{\jS}.
\end{cases}
\end{equation}

\begin{proposition}
\label{p:3}
In the context of Problem~\ref{prob:5},
let $x_{0}\in L^2(\upOmega,\FE,\PP;\HS)$ and let
$(\upgamma_{\nS})_{\nnn}$ be a sequence in $\RPP$ such that
$\sum_{\nnn}\upgamma_{\nS}=\pinf$ and
$\sum_{\nnn}\upgamma_{\nS}^2<\pinf$. Iterate
\begin{equation}
\label{e:algo3}
\hskip -1mm 
\begin{array}{l}
\text{for}\;\nS=0,1,\ldots\\
\left\lfloor
\begin{array}{l}
\textup{take}\;(i_{\nS},j_{\nS})\;\textup{uniformly
in}\;\mathsf{K}_1\times\mathsf{K}_2\;\textup{and independent of}\;
\upsigma(x_{0},\ldots,x_{\nS})\\
x_{\nS+1}=\prox_{\upgamma_{\nS}\mathsf{f}_{i_{\nS}^{}}}
\brk1{x_{\nS}-\upgamma_{\nS}\nabla\gS_{j_{\nS}}(x_{\nS})}.
\end{array}
\right.\\
\end{array}
\end{equation}
Denote by $\ZS$ the set of solutions to Problem~\ref{prob:5} and
assume that $\ZS\neq\emp$.
Then $(x_{\nS})_{\nnn}$ converges $\Pas$ and in
$L^1(\upOmega,\FE,\PP;\HS)$ to a $\ZS$-valued random variable.
\end{proposition}
\begin{proof}
We deduce from \eqref{e:ap6} that, for every $\xS\in\HS$ and
$\jS\in\mathsf{K}_2$,
\begin{equation} 
\norm{\nabla^2\gS_{\jS}(\xS)}
=\dfrac{(1-\upalpha)
\exp\brk1{\upxi_{\jS}\scal{\xS}{\mathsf{u}_{\jS}}}}
{\brk2{1+\exp\brk1{\upxi_{\jS}\scal{\xS}{\mathsf{u}_{\jS}}}}^2}
\norm{\mathsf{u}_{\jS}}^2
\leq\dfrac{1-\upalpha}{4}\norm{\mathsf{u}_{\jS}}^2.
\end{equation}
Set
\begin{equation}
\upbeta=\dfrac{1-\upalpha}{4}\max_{\jS\in\mathsf{K}_2}
\norm{\mathsf{u}_{\jS}}^2.
\end{equation}
Hence, for every
$\kS\in\mathsf{K}_2$, $\nabla\gS_{\kS}$ is 
$\upbeta$-Lipschitzian. Thus, we confirm that Problem~\ref{prob:5}
is an instance of Problem~\ref{prob:1}, and
\eqref{e:algo3} is an instance of Algorithm~\ref{algo:1}.
It follows from Fermat's rule \cite[Theorem~16.3]{Livre1} and
\cite[Theorems~1.37 and 3.8]{Penn24} that there exists
$\bar{\zS}\in\ZS$ and, for every $\kS\in\mathsf{K}_1$,
$\sS_{\kS}\in\partial\fS_{\kS}(\bar{\zS})$ such that 
\begin{equation}
\label{e:144}
0=\int_{\upOmega}\brk1{\sS_{k(\upomega)}+
\nabla\gS_{k(\upomega)}(\bar{\zS})}\PP(d\upomega).
\end{equation}
Furthermore, we infer from \eqref{e:p51} and the fact that the sets
$\mathsf{K}_1$ and $\mathsf{K}_2$ are finite that the subgradients
of $(\fS_{\kS})_{\kS\in\mathsf{K}_1}$ and the gradients of
$(\gS_{\kS})_{\kS\in\mathsf{K}_2}$ are uniformly bounded.
Hence Assumptions~\ref{as:1} and \ref{as:2} hold with $\uppsi$
constant equal to the uniform bound.
Thus the conclusion follows from Theorems~\ref{t:1}\ref{t:1iii} and
\ref{t:1}\ref{t:1v}.
\end{proof}

\section{Application to inconsistent convex feasibility problems}
\label{sec:5}

We apply the stochastic proximal gradient method to the
inconsistent convex feasibility problem. 

\begin{problem}
\label{prob:4}
Let $(\mathsf{K},\EuScript{K})$ be a measurable space and let
$k\colon(\upOmega,\FE,\PP)\to(\mathsf{K},\EuScript{K})$ be a random
variable. Let $\CS$ be a nonempty closed convex subset of $\HS$,
and, for every $\kS\in\mathsf{K}$, let $\ZS_{\kS}$ be a nonempty
closed convex subset of $\HS$. It is assumed that the mapping
\begin{equation}
\boldsymbol{\mathsf{T}}
\colon(\mathsf{K}\times\HS,
\EuScript{K}\otimes\BE_{\HS})\to(\HS,\BE_{\HS})
\colon(\kS,\xS)\mapsto\dS_{\ZS_{\kS}}^2\xS
\end{equation}
is measurable, $\EE\boldsymbol{\mathsf{T}}(k,0)<\pinf$, and
$\Argmin\EE\boldsymbol{\mathsf{T}}(k,\cdot)\neq\emp$. The task is
to 
\begin{equation}
\label{e:prob2}
\minimize{\xS\in\CS}{\int_{\upOmega}\dfrac{1}{2}
\dS_{\ZS_{k(\upomega)}}^2(\xS)\PP(d\upomega)}.
\end{equation}
\end{problem}

Minimizing the integral of the squared distances dates back
to the expected-projection method \cite{Butn95b,Butn95}. However,
this method requires the activation of every set at every iteration
via a Bochner integral average. For the consistent case, random
iterative methods have been proposed; see
\cite[Remark~5.6]{Moco26}. These methods activate only a finite
number of sets at each iteration and guarantee convergence to a
solution. For the inconsistent case, the random iterative method of
\cite{Herm23} guarantees convergence in distribution to an
invariant measure by randomly selecting one set at every iteration.
Stronger modes of convergence have not been shown for the
inconsistent case. As an application of Theorem~\ref{t:1}, we
introduce a randomized single-set activation algorithm for
solving Problem~\ref{prob:4} that converges
both almost surely and in $L^1(\upOmega,\FE,\PP;\HS)$.

\begin{proposition}
\label{p:2}
In the context of Problem~\ref{prob:4},
let $x_{0}\in L^2(\upOmega,\FE,\PP;\HS)$ and let
$(\upgamma_{\nS})_{\nnn}$ be a sequence in $\RPP$ such that
$\sum_{\nnn}\upgamma_{\nS}=\pinf$ and
$\sum_{\nnn}\upgamma_{\nS}^2<\pinf$. Iterate
\begin{equation}
\label{e:algo2}
\hskip -1mm 
\begin{array}{l}
\text{for}\;\nS=0,1,\ldots\\
\left\lfloor
\begin{array}{l}
k_{\nS}\;\textup{is a copy of $k$ and is independent of}\;
\upsigma(x_{0},\ldots,x_{\nS})\\
x_{\nS+1}=\proj_{\CS}
\brk1{(1-\upgamma_{\nS})x_{\nS}
+\upgamma_{\nS}\proj_{\ZS_{k_{\nS}}}(x_{\nS})}.
\end{array}
\right.\\
\end{array}
\end{equation}
Denote by $\ZS$ the set of solutions to Problem~\ref{prob:4}.
Then $(x_{\nS})_{\nnn}$ converges $\Pas$ and in
$L^1(\upOmega,\FE,\PP;\HS)$ to a $\ZS$-valued random variable.
\end{proposition}
\begin{proof}
Let us define
\begin{equation}
(\forall\kS\in\mathsf{K})\quad
\begin{cases}
\fS_{\kS}=\iota_{\CS}\in\upGamma_{0}(\HS);\\
\gS_{\kS}=\dfrac{1}{2}\dS_{\ZS_{\kS}}^2.
\end{cases}
\end{equation}
We deduce from 
\cite[Example~12.25 and Corollary~12.31]{Livre1} that 
\begin{equation}
(\forall\kS\in\mathsf{K})\quad
\begin{cases}
(\forall\nnn)\;\;
\prox_{\upgamma_{\nS}\fS_{\kS}}=\proj_{\CS};\\[1mm]
\nabla\gS_{\kS}=\Id-\proj_{\ZS_{\kS}}.
\end{cases}
\end{equation}
It then follows from \cite[Corollary~4.18]{Livre1} that, for every
$\kS\in\mathsf{K}$, $\nabla\gS_{\kS}$ is firmly nonexpansive, hence
\mbox{$1$-Lipschitzian}. This confirms that Problem~\ref{prob:4} is
an instance of Problem~\ref{prob:1} with $\upbeta=1$, and
\eqref{e:algo2} is an instance of Algorithm~\ref{algo:1}.
We deduce from Fermat's rule \cite[Theorem~16.3]{Livre1} and
\cite[Theorems~1.37 and 3.8]{Penn24} that there exists
$\bar{\zS}\in\ZS$ and
$\sS\in\mathsf{N}_{\CS}(\bar{\zS})$ such that 
\begin{equation}
\label{e:122}
0=\sS+\int_{\upOmega}
\nabla\gS_{k(\upomega)}(\bar{\zS})\PP(d\upomega).
\end{equation}
Moreover, for every $\xS\in\HS$, 
\begin{equation}
\label{e:123}
\EE\norm{\nabla\gS_{k}(\xS)}^2
=\EE\dS_{\ZS_{k}}^2(\xS)
\leq 2\EE\dS_{\ZS_{k}}^2(0)+2\norm{\xS}^2
=2\EE\boldsymbol{\mathsf{T}}(k,0)+2\norm{\xS}^2.
\end{equation}
Combining \eqref{e:122} and \eqref{e:123} with $\xS=\bar{\zS}$, we
deduce that Assumption~\ref{as:1} holds. On the other hand, for
every $\kS\in\mathsf{K}$ and every $\xS\in\CS$,
$\moyo{\partial\fS_{\kS}(\xS)}{0}=\moyo{\mathsf{N}_{\CS}(\xS)}{0}
=0$. Hence Assumption~\ref{as:2} also holds by setting
$\uppsi=\max\{2,2\EE\boldsymbol{\mathsf{T}}(k,0)\}
(1+\abs{\cdot}^2)$ in \eqref{e:123}. Therefore the conclusion
follows from Theorems~\ref{t:1}\ref{t:1iii} and
\ref{t:1}\ref{t:1v}. 
\end{proof}

\section*{Acknowledgement}

This work is part of the author's Ph.D. dissertation.
The author gratefully acknowledges the guidance of his Ph.D.
advisor P. L. Combettes throughout this work.

\end{document}